\pdfoutput=1
\documentclass[a4paper]{article}
\usepackage[T2A]{fontenc}
\usepackage[cp1251]{inputenc}
\usepackage[russian,english]{babel}
\usepackage[tbtags]{amsmath}
\usepackage{amsfonts,amssymb,mathrsfs,amscd}
 \usepackage{msb-a}
 
\newlength\tindent
\renewcommand{\indent}{\hspace*{\tindent}}
\JournalName{}
\numberwithin{equation}{section}
\theoremstyle{plain}

\newtheorem{theorem}{Theorem}
\theoremstyle{plain}
\newtheorem{theoremm}{Theorem}
\newtheorem{lemma}{Lemma}

\theoremstyle{definition}
\newtheorem{proof}{}

\newtheorem{remark}{Remark}
\renewcommand{\refname}{References}
\begin{document}

\title{Omega-theorems for the Riemann zeta function and its derivatives near the line $\mathrm{Re}\,s=1$.}
\author{Alexander Kalmynin}
\address{National Research University Higher School of Economics, Russian Federation, Math Department, International Laboratory of Mirror Symmetry and Automorphic Forms}
\email{alkalb1995cd@mail.ru}
\date{}
\udk{}
\maketitle
\renewcommand{\abstract}{\begin{bf}{Abstract.}\end{bf}}
\markright{Omega-theorems for the Riemann zeta function}

\footnotetext[0]{The author is partially supported by Laboratory of Mirror Symmetry NRU HSE, RF Government grant, ag. № 14.641.31.0001, the Simons Foundation and the Moebius Contest Foundation for Young Scientists}

\begin{abstract}
We introduce a generalization of the method of S.\,P.\,Zaitsev \cite{Z}. This generalization allows us to prove omega-theorems for the Riemann zeta function and its derivatives in some regions near the line $\mathrm{Re}\,s=1$.
\end{abstract}
\begin{fulltext}

\section{Introduction}
It is well known that the estimates for the absolute value of the Riemann zeta function $\zeta(s)$ inside the critical strip $0<\mathrm{Re}\,s<1$ have many applications in number theory (see, for example \cite{VK},\cite{Ivic}). At the same time, there exist a number of omega-theorems which show that the quantity $|\zeta(s)|$ can attain very large values in this domain. For example, E.\,C.\,Titchmarsh \cite{Tit} showed that for any fixed $0<\sigma<1$ and arbitrary $\varepsilon>0$ the relation
$$|\zeta(\sigma+it)|=\Omega(\exp(\varphi(t))),\,|t| \to \infty$$
\noindent holds, where $\varphi(t)=(\log |t|)^{1-\sigma-\varepsilon}$.
Later,  N.\,Levinson \cite{Lev} (see also \cite{V}) and H.\,L.\,Montgomery \cite{Montg} proved stronger propositions corresponding to
$$\varphi(t)=c_1(\sigma)\frac{(\log|t|)^{1-\sigma}}{\log\log|t|} \text{ and } \varphi(t)=c_2(\sigma)\frac{(\log|t|)^{1-\sigma}}{(\log\log|t|)^{\sigma}}$$
\noindent respectively ($c_1$ and $c_2$ are some positive constants dependent on $\sigma$).
In the case $\sigma=1$ it is known (cf. \cite{Littlewood}) that 
$$|\zeta(1+it)|=\Omega(\log\log|t|)$$ 
and under assumption of the Riemann hypothesis the same upper bound is true.
\\
In this work, we will study the large values of $|\zeta(s)|$ in the domains of the form \\$\{s=\sigma+it:\,\sigma(t)\leq\sigma<1\}$, where $\sigma(t)$ is some function which tends to 1 monotonically as $|t|\to \infty$. Our results generalize the following theorem by S.\,P.\,Zaitsev \cite{Z}:
\begin{theoremm}
Let $T$ and $\varepsilon$ be positive real numbers. Denote
$$\Sigma_\varepsilon(T)=\{s=\sigma+it : 1-(4+\varepsilon)\frac{\log\log\log t}{\log\log t}\leq \sigma\leq 1, t_0\leq |t|\leq T\},$$
\noindent where $t_0$ is some fixed positive number. Then the inequality
$$\limsup_{\substack{s\in\Sigma_\varepsilon(T) \\ T\to \infty}} \frac{|\zeta(s)|}{\log T}\geq 1$$
holds.
\end{theoremm}

We will generalize the method of proof devised by the author of \cite{Z} and show that the following proposition is true:

\begin{theorem}
Let $n\geq 0$ be some fixed integer and $\varepsilon,\delta,A,\alpha$ be some positive constants with $0<\alpha<1$, $A>0$, $\varepsilon>0$ and $\delta>0$. Then for any pair of functions $(F,\sigma)$, where
$$F(t)=\exp((\log\log t)^{1+\varepsilon/2-\delta}),\, \sigma(t)=1-(4+\varepsilon)\frac{\log\log\log t}{\log\log t},$$
$$F(t)=\exp\exp\left(\frac{\log\log t}{\log\log\log t}\right),\, \sigma(t)=1-\frac{2+o(1)}{\log\log\log t},$$
$$F(t)=(\log\log t)^A,\, \sigma(t)=1-(2+o(1))\frac{\log\log\log t}{\log\log t}$$
or
$$F(t)=\exp\exp((\log\log t)^\alpha),\, \sigma(t)=1-\frac{2+o(1)}{(\log\log t)^{1-\alpha}}$$
\noindent and subset
$$\Sigma(\sigma,T)=\{s=\sigma+it : \sigma(t)\leq \sigma\leq 1, t_0(\sigma)\leq t\leq T\}\subset \mathbb C$$
\noindent we have the inequality
$$\limsup_{\substack{s\in\Sigma(\sigma,T)\\ T\to \infty}} \frac{|\zeta^{(n)}(s)|}{F(t)}\geq 1,$$
\noindent where $t_0(\sigma)$ is a positive real number that depends only on $\sigma$ and decrease in all the $o$ symbols depends on $\varepsilon,A,\alpha,\delta$ and $n$.
\end{theorem}

\section{Auxiliary results}

In this section, we prove two theorems about the coefficients of Dirichlet series which will be used to prove the Theorem 1. Both propositions can be of interest by themselves and, due to the high level of generality, also applicable in other situations.

The first theorem gives an estimate for the coefficients of Dirichlet series in terms of its singularities and order of growth in some region inside the critical strip:
\begin{theorem}

Let $f(s)=\sum\limits_{n=1}^{+\infty} a_n n^{-s}$ be some Dirichlet series with nonnegative coefficients which is absolutely convergent in the domain $\mathrm{Re}\,s>1$. Suppose that $f$ admits an analytic continuation to the region
$$\sigma\geq \sigma(|t|)>\delta>0, s\neq 1$$
and in the neighbourhood of $s=1$ the inequality
$$f(s)\ll F\left(\frac{1}{|s-1|}\right)$$
\noindent holds, where $\sigma:\mathbb R_{>0}\to\mathbb R_{>0}$ is a continuous nondecreasing function such that for large enough $x$ we have $\sigma(x)\leq 1-\frac{2\log\log x}{\log x}$ and $F(x)$ is a postive increasing functions which grows at least as fast as $x$.

If for all sufficiently large $t$ and $\sigma\geq\sigma(|t|)$ we have $$|f(\sigma+it)|\ll |t|,$$ then the bound
$$\sup_{n \leq x} a_n \ll x^{(\sigma(x)+1)/2}\sqrt{\log x}\sqrt{F_{\infty}(2x)}$$
is true, where $F_\infty(x)=\inf\limits_{\varepsilon>0} x^\varepsilon F\left(\frac{1}{\varepsilon}\right)$.

\end{theorem}

Proof of this proposition requires four more lemmas. We begin with the truncated Perron's formula:

\begin{lemma}

Let $b_1,b_2,\ldots$ be a sequence of complex numbers, $B(x)=\sum\limits_{n\leq x} b_n$ and $F_B(s)=\sum\limits_{n=1}^{+\infty} b_n n^{-s}$. Assume that the series in the definition of $F_B(x)$ is absolutely convergent for any $\mathrm{Re}\,s>1$. Then for any $b>1, x\geq 2, T\geq 2$ we have the equality
$$\int_0^x B(\xi)d\xi=\frac{1}{2\pi i}\int_{b-iT}^{b+iT}\frac{F_B(s)}{s(s+1)}x^{s+1}ds+R(b,x,T),$$
where
$$R(b,x,T)\ll \frac{x^{b+1}}{T}\int_1^{+\infty} \frac{|B(\xi)|}{\xi^{b+1}}d\xi+2^b\left(\frac{x\log x}{T}+\log T\right)\max_{x/2\leq \xi \leq 3x/2}|B(\xi)|.$$
\end{lemma}

\begin{proof}[Proof] See \cite{VK}, Appendix, \S 5.
\end{proof}

The next estimate for the summatory function of coefficients of Dirichlet series easily follows from the conditions of the Theorem 2:

\begin{lemma}
Suppose that the series $f(s)=\sum\limits_{n=1}^{+\infty} a_n n^{-s}$ satisfy the assumtions of the Theorem 2. Then the inequality
$$A(x)=\sum_{n \leq x} a_n\ll xF_{\infty}(x)$$
holds.
\end{lemma}

\begin{proof}[Proof]
Since for any $\varepsilon>0$ the series $f(1+\varepsilon)=\sum\limits_{n=1}^{+\infty} a_n n^{-1-\varepsilon}$ converges, by the nonnegativity of $a_n$ we get
$$0\leq A(x)=\sum_{n \leq x} a_n \leq \sum_n a_n\left(\frac{x}{n}\right)^{1+\varepsilon} \leq f(1+\varepsilon)x^{1+\varepsilon}\ll x^{1+\varepsilon}F\left(\frac{1}{\varepsilon}\right).$$
Choosing $\varepsilon$ such that $x^{\varepsilon}F\left(\frac{1}{\varepsilon}\right)=F_\infty(x)$, we obtain the required inequality.
\end{proof}

By the means of so-called ``asymptotic differentiation'', the Lemma 3 allows us to deduce the remainder term in the asymptotic formula for $A(x)$ in terms of remainder term in formula for $\int\limits_0^x A(\xi)d\xi$:

\begin{lemma}
Let $f(s)=\sum\limits_{n=1}^{+\infty} a_n n^{-s}$ be some Dirichlet series which meets the hypotheses of the Theorem 2. Suppose that for all large enough $x$ the equality
$$\int_0^x A(\xi)d\xi=\mathrm{Res}_{s=1} \frac{f(s)x^{s+1}}{s(s+1)}+O(Q(x))$$
holds, where $0<Q(x)\leq \frac{x^2}{4}$ is an increasing function. Then we have
$$A(x)=\mathrm{Res}_{s=1}\frac{f(s)x^s}{s}+O(\sqrt{F_{\infty}(2x)}\sqrt{Q(2x)}).$$
\end{lemma}

\begin{proof}[Proof]
As the function $A(\xi)$ is nondecreasing, for any $0<h\leq x$ we have
$$\frac{1}{h}\int_{x-h}^x A(\xi)d\xi \leq A(x) \leq \frac{1}{h}\int_{x}^{x+h} A(\xi)d\xi.$$
By virtue of our lemma, the inequalities
$$\mathrm{Res}_{s=1}\frac{f(s)(x^{s+1}-(x-h)^{s+1})}{hs(s+1)}+O\left(\frac{Q(2x)}{h}\right)\leq A(x)$$
and
$$A(x)\leq \mathrm{Res}_{s=1}\frac{f(s)((x+h)^{s+1}-x^{s+1})}{hs(s+1)}+O\left(\frac{Q(2x)}{h}\right)$$
are true. On the other hand,
$$\frac{(x+h)^{s+1}-x^{s+1}}{hs(s+1)}=\int_x^{x+h} \frac{\xi^s}{hs}d\xi=\frac{x^s}{s}+\int_x^{x+h} \frac{\xi^s-x^s}{hs}d\xi=$$
$$=\frac{x^s}{s}+\int_x^{x+h} \frac{1}{h}\left(\int_x^\xi \theta^{s-1}d\theta\right)d\xi=\frac{x^s}{s}+O(h(2x)^{\sigma-1})$$
and similarly
$$\frac{x^{s+1}-(x-h)^{s+1}}{hs(s+1)}=\frac{x^s}{s}+O(h(2x)^{\sigma-1}).$$
Hence, for any $1\gg\varepsilon>0$ the equalities
$$\mathrm{Res}_{s=1}\frac{f(s)((x+h)^{s+1}-x^{s+1})}{hs(s+1)}=\frac{1}{2\pi i} \int_{|s-1|=\varepsilon} f(s)\left(\frac{x^s}{s}+O(h(2x)^{\varepsilon})\right)=$$
$$=\mathrm{Res}_{s=1} \frac{f(s)x^s}{s}+O(h(2x)^{\varepsilon}\sup_{|s-1|=\varepsilon} |f(s)|)$$
and
$$\mathrm{Res}_{s=1}\frac{f(s)(x^{s+1}-(x-h)^{s+1})}{hs(s+1)}=\mathrm{Res}_{s=1}\frac{f(s)x^s}{s}+O(h(2x)^{\varepsilon}\sup_{|s-1|=\varepsilon}|f(s)|)$$
hold. Furthermore, the equality $\sup\limits_{|s-1|=\varepsilon} |f(s)|=O(F(1/\varepsilon))$ holds. Therefore, choosing $\varepsilon$ optimally we find
$$A(x)=\mathrm{Res}_{s=1}\frac{f(s)x^s}{s}+O\left(hF_{\infty}(2x)+\frac{Q(2x)}{h}\right).$$
Substituting  $h=\sqrt{\frac{Q(2x)}{F_{\infty}(2x)}}$ in this equality (this choice is admissible as\\ $0<\sqrt{Q(2x)}\leq x$) we obtain the desired result.
\end{proof}

\begin{lemma}

Under the conditions of the Theorem 2 the equalities
$$\int_0^x A(\xi)d\xi=\mathrm{Res}_{s=1} \frac{f(s)x^{s+1}}{s(s+1)}+O(x^{\sigma(x)+1}\log x)$$
and
$$A(x)=\mathrm{Res}_{s=1} \frac{f(s)x^s}{s}+O(x^{(\sigma(x)+1)/2}\sqrt{F_{\infty}(2x)}\sqrt{\log x})$$
are true.
\end{lemma}

\begin{proof}[Proof]
Choose $T=x, b=1+\delta$, where $\delta$ is the constant from the formulation of the Theorem 2. Applying Lemma 1 we get
$$\int_0^x A(\xi)d\xi=\frac{1}{2\pi i}\int_{1+\delta-ix}^{1+\delta+ix} \frac{f(s)x^{s+1}}{s(s+1)}ds+R(b,x,x)$$
with
$$R(b,x,x)\ll x^{1+\delta}\int_1^{+\infty} \frac{|A(\xi)|}{\xi^{\delta+2}}d\xi+(\log x)\max_{x/2\leq\xi \leq 3x/2}|A(\xi)|.$$
Let us estimate the first summand. According to the Lemma 2 we have $A(\xi)=O(\xi^{1+\delta/2})$. Hence,
$$\int_1^{+\infty} \frac{|A(\xi)|}{\xi^{\delta+2}}d\xi=O(1).$$
Therefore,
$$x^{1+\delta}\int_1^{+\infty}\frac{|A(\xi)|}{\xi^{\delta+2}}d\xi=O(x^{1+\delta}).$$
As for the second summand, it equals $O(xF_{\infty}(3x)\log x)$ by the Lemma 2. Consequently, we have
$$R(b,x,x)=O(x^{1+\delta}+xF_{\infty}(3x)\log x).$$
It remains to calculate the integral. To do this, let us move the contour to the curve $\mathrm{Re}\,s=\sigma(|t|)$. By Cauchy's integral formula, we get
$$\frac{1}{2\pi i}\int_{1+\delta-ix}^{1+\delta+ix} \frac{f(s)x^{s+1}}{s(s+1)}ds=\mathrm{Res}_{s=1} \frac{f(s)x^{s+1}}{s(s+1)}+\frac{1}{2\pi i}(I_0-I_{-}+I_{+}),$$
\noindent where
$$I_0=\int_{\substack{|t| \leq x\\ \sigma=\sigma(|t|)}} \frac{f(s)x^{s+1}}{s(s+1)}ds,$$
$$I_+=\int_{\sigma(x)+ix}^{1+\delta+ix} \frac{f(s)x^{s+1}}{s(s+1)}ds$$
and
$$I_-=\int_{\sigma(x)-ix}^{1+\delta-ix} \frac{f(s)x^{s+1}}{s(s+1)}ds.$$
Due to the fact that $\sigma(t)$ is nondecreasing, the estimate
$$|I_0|\leq \int_{\substack{|t| \leq x \\ \sigma \geq \sigma(|t|)}} \frac{|f(s)|x^{\sigma(x)+1}}{|s(s+1)|}d|s|$$
holds. Note that on the contour we have $|s(s+1)| \gg (|t|+1)^2$. Furthermore, if $|t|$ is large enough, then due to the conditions of the Theorem 2 for $s$ lying on our contour the inequality
$$|f(s)| \ll |t|+1$$
is true. If $|t|\leq t_1$, then $\sigma(|t|)$ is bounded away from 1, hence the quantity $|f(s)|$ is bounded uniformly in $x$ therefore for the rest of the numbers $s$ on the contour we have the same inequality. Consequently,
$$|I_0| \leq \int_{|t| \leq x} \frac{x^{\sigma(x)+1}}{|t|+1}dt \ll x^{\sigma(x)+1}\log x.$$
The values $I_+$ and $I_-$ are conjugate complex numbers. So, it suffices to estimate one of them.
$$I_+ \ll \int_{\sigma(x)}^{1+\delta} \frac{|f(s)|x^{\sigma+1}}{x^2}d\sigma \ll x^{\delta+1}.$$
As $\sigma(x)>\delta$, we obtain the equality
$$\frac{1}{2\pi i} \int_{1+\delta-ix}^{1+\delta+ix} \frac{f(s)x^{s+1}}{s(s+1)}ds=\mathrm{Res}_{s=1} \frac{f(s)x^{s+1}}{s(s+1)}+O(x^{\sigma(x)+1}\log x).$$

 Furthermore, we have $x^{\delta+1}+xF_{\infty}(3x)\log x \ll x^{\delta+1} \log x \ll x^{\sigma(x)+1}\log x$. Hence,
$$\int_0^x A(\xi)d\xi=\mathrm{Res}_{s=1} \frac{f(s)x^{s+1}}{s(s+1)}+O(x^{\sigma(x)+1}\log x),$$
\noindent which was to be proved.

The second inequality follows form the Lemma 3 with $Q(x)=x^{1+\sigma(x)}\log x$, as

\noindent $\sigma(x) \leq 1-\frac{2\log\log x}{\log x}$ and so $Q(x)\ll \frac{x^2}{\log x} \leq \frac{x^2}{4}$ for all $x$ large enough.
\end{proof}

The Theorem 2 is easily deduced from the last equality of the Lemma 4:

\begin{proof}[Proof of Theorem 2]

Let us notice that $a_n=A(n)-A(n-1)$. The Lemma 4 implies that for any $1<n\leq x$ we have
$$a_n=\mathrm{Res}_{s=1} \frac{f(s)(n^s-(n-1)^s)}{s}+O(x^{(\sigma(x)+1)/2}\sqrt{\log x}\sqrt{F_{\infty}(2x)}).$$
But the first summand is negligible, because $$n^s-(n-1)^s\ll |s|n^{\sigma-1}$$ and so
$$\mathrm{Res}_{s=1}\frac{f(s)(n^s-(n-1)^s)}{s}=\frac{1}{2\pi i}\int_{|s-1|=\delta} \frac{f(s)(n^s-(n-1)^s)}{s}ds \ll n^\delta.$$
This concludes the proof of the theorem.
\end{proof}

In the next theorem we construct a natural family of Dirichlet series with large coefficients.
We begin with the following lemma:
\begin{lemma}

Let $F:\mathbb R_{\geq 0}\to \mathbb R_{\geq 0}$ be an increasing function with
$$\lim_{x \to \infty} \frac{F(x)}{x}=+\infty.$$
Denote $F^*(y)=\sup_{x \geq 0}(xy-F(x))$ and $\log_{+} x=\max(0,\log x)$.
Then the series
$$M_F(z)=\sum_{n=1}^{+\infty} \frac{z^n}{n^2}e^{-F^*(n)}$$
converges absolutely for any $z \in \mathbb C$, defines an entire function and satisfies the inequality
$$|M_F(z)|\leq \frac{\pi^2}{6} e^{F(\log_{+}|z|)}.$$
\end{lemma}
\begin{remark}
If the function $\frac{F(x)}{x}$ is bounded, then any entire function satisfying the last inequality is a polynomial, while for our subsequent constructions we need Taylor coefficients to be positive.
\end{remark}

\begin{proof}[Proof]

Indeed, by the definition of $F^*$, we have 

$n\log_{+}|z|-F(\log_{+}|z|)\leq F^*(n)$. Hence, for any $z$ one has
$$|z^ne^{-F^*(n)}| \leq e^{nF(\log_{+}|z|)-F^*(n)}\leq e^{F(\log_{+}|z|)}.$$
Consequently,
$$|M_F(z)|\leq \sum_{n=1}^{+\infty} \frac{|z^ne^{-F^*(n)}|}{n^2} \leq e^{F(\log_{+}|z|)}\sum_{n=1}^{+\infty} \frac{1}{n^2}=\frac{\pi^2}{6}e^{F(\log_{+}|z|)}.$$
\end{proof}
With the help of this construction we prove the following fact:
\begin{theorem}
Let $g(s)=\sum\limits_{n=1}^{+\infty} g_n n^{-s}$ be some Dirichlet series which converges absolutely for any $s$ with $\mathrm{Re}\,s>1$. Suppose that for any $n$ we have $1\leq g_n$. Then the series $M_F(g(s))$ also converges for $\mathrm{Re}\,s>1$ and for large enough $x$ the inequality
$$\sup_{n \leq x} a_n \geq e^{\log k\frac{\log x}{\log\log x}-F^*(k)}$$
holds for any natural number $k$.
\end{theorem}

\begin{proof}[Proof]

Indeed,
$$M_F(g(s))=\sum_{n=1}^{+\infty} \frac{g^n(s)}{n^2}e^{-F^*(n)}.$$
Convergence of this series follows from the Lemma 5. As for any $n$ the Dirichlet coefficients of the function $g^n(s)$ are positive, Dirichlet coefficients of $M_F(g(s))$ are positive, too.

To prove the lower bound for the coefficients, let us note that for any natural $k$ we have
$$a_n \geq \frac{g_n(k)e^{-F^*(k)}}{k^2},$$
where $\sum\limits_{n=1}^{+\infty} \frac{g_n(k)}{n^s}=g^k(s)$. Now, for sufficiently large real $x$ choose a real number $m$ satisfying the inequalities
$$\pi(m)=\sum_{p \leq m} 1\geq \frac{\log x}{\log\log x}+2$$
and
$$\theta(m)=\sum_{p \leq m} \log p \leq \log x.$$
Such a choice is possible because of the formulas $$\pi(m)=\frac{m}{\log m}+\frac{m}{\log^2 m}+O\left(\frac{m}{\log^3 m}\right)$$ and 
$$\theta(m)=m+O(me^{-c\sqrt{\log m}}).$$ Choosing $n=\prod\limits_{p \leq m} p=e^{\theta(m)} \leq x$, we get
$$g_n(k)=\sum_{n_1\ldots n_k=n} g_{n_1}\ldots g_{n_k} \geq \sum_{n_1\ldots n_k=n}1=k^{\pi(m)}\geq k^{2+\frac{\log x}{\log\log x}}=k^2e^{\log k \frac{\log x}{\log\log x}}.$$
From here we deduce the relation
$$\sup_{l \leq x} a_l \geq a_n \geq \frac{g_n(k)e^{-F^*(k)}}{k^2} \geq e^{\log k\frac{\log x}{\log\log x}-F^*(k)},$$
which was to be proved.
\end{proof}

\section{Proof of Theorem 1}
Theorem 3 together with the Lemma 5 allows us to construct Dirichlet series which grows moderately in some subset inside their domain of analyticity with coefficients that can attain rather large values. On the other hand, general Theorem 2 gives an upper bound for the coefficients of Dirichlet series which do not grow too fast in some region of the complex plane.

Assuming that the derivative $\zeta^{(n)}(s)$ does not satisfy the Theorem 1 in some region to the left of the line $\sigma=1$, we will construct Dirichlet series which takes no large values in this domain. However, some of its coefficients will be big enough to contradict the Theorem 2.

More precisely, suppose that the inequality
\begin{equation}
\label{eq1}
\limsup_{\substack{s \in \Sigma(\sigma,T) \\ T \to \infty}} \frac{|\zeta^{(n)}(s)|}{F(t)} \geq 1
\end{equation}

is false. Then for some $c$ with condition $0<c<1$ and for every $\sigma$ and $t$ from the domain $\Sigma(\sigma,T)$ with sufficiently large $|t|$ and $T$ we have
$$\frac{|\zeta^{(n)}(s)|}{F(|t|)} \leq 1-c.$$
From this for $|t| \geq t_0$ we find
$$\frac{|1+\zeta^{(n)}(s)|}{F(|t|)} \leq 1-c+\frac{1}{F(|t|)} \leq 1-\frac{c}{2}<1.$$
Hence, for such $\sigma$ and $t$ the inequalities
$$|1+\zeta^{(n)}(\sigma+it)|<F(|t|), \, F^{-1}(|1+\zeta^{(n)}(\sigma+it)|)<t$$
hold. Let now $G(u)$ be some real monotonically increasing function that satisfies the relation 
 $$e^{G(\log_{+}|z|)}\leq F^{-1}(|z|).$$ 
 Substituting $z=1+\zeta^{(n)}(s)$ to the last inequality, we obtain 
 $$e^{G(\log_{+}|z|)}\leq F^{-1}(|1+\zeta^{(n)}(s)|)<|t|.$$ 
 Consequently, setting 
 $$f(s)=M_G(1+\zeta^{(n)}(s)),$$
we deduce for all $s \in \Sigma(\sigma,T)$ the inequality $|f(\sigma+it)|\ll |t|$. Moreover, due to the Theorem 3 for the coefficients $a_n$ of Dirichlet series $f(s)$ the lower bound
\begin{equation}
\label{eq2}
\sup_{n\leq x} a_n \geq \exp((\log k)\frac{\log x}{\log\log x}-F^*(k))
\end{equation}

holds for arbitrary positive integer $k$. However, by the Theorem 2 the estimate
\begin{equation}
\label{eq3}
\sup_{n \leq x} a_n \ll x^{(1+\sigma(x))/2}\sqrt{\log x}\sqrt{F_\infty(2x)}.
\end{equation}
is true. Now, for any pair $F,\sigma$ an appropriate choice of function $G$ and natural number $k$ will lead the bounds~\ref{eq2} and~\ref{eq3} to contradiction.

\begin{proof}[Proof of Theorem 1]
Let $$F(t)=\exp((\log\log t)^{1+\varepsilon/2-\delta})$$ and $$\sigma(t)=1-(4+\varepsilon)\frac{\log\log\log t}{\log\log t}.$$ 

Choose $G(z)=\exp\exp(\frac{\log z}{1+\varepsilon/2-\delta})$ (we define this function by this formula on the interval $[1,+\infty]$ and extend it to the interval $[0,1]$ monotonically). Suppose that for the pair $(F,\sigma)$ Theorem 1 is false. Consider the function\\ $f(s)=M_G(\zeta^{(n)}(s)+1).$ For all sufficiently large $t$ it satisfies for $\sigma\geq\sigma(t)$ the inequality $|f(s)|\ll |t|$ as $|\zeta^{(n)}(s)|\leq c\exp((\log\log t)^{1+\varepsilon/2-\delta})$ for some $c<1$ and hence
$$f(s)\leq e^{G(\log_+(|\zeta^{(n)}(s)|+1))}\ll e^{G((\log\log t)^{1+\varepsilon/2-\delta})}=\exp\exp\exp(\log\log\log t)=t.$$
Furthermore, for $s\to 1$ we have
$$f(s)\ll \exp\exp\left(\frac{1}{|s-1|}\right).$$
Indeed, in the neighbourhood of $s=1$ the estimate
$$\zeta^{(n)}(s) \ll \frac{1}{|s-1|^{n+1}},$$
holds. Thus, for any natural $m$ we have
$$|M_G(z)|\ll \exp(\exp(|z|^{1/m})).$$
Consequently, for $s \to 1$ we have
$$|M_G(1+\zeta^{(n)}(s))| \ll \exp\exp\left(\frac{1}{|s-1|^{1-1/(n+2)}}\right) \ll \exp\exp\left(\frac{1}{|s-1|}\right).$$
Thus, the estimate
$$\sup_{n \leq x} a_n \ll x^{(\sigma(x)+1)/2}\sqrt{\log x}R_{\infty}(2x)$$
holds, where $a_n$ are the coefficients of $M_G(\zeta^{(n)}(s))$ and $R(x)=e^{e^x}$.

On the other hand, by the Theorem 3, for any positive integer $k$ the inequality
$$\sup_{n \leq x} a_n \gg x^{\log k \frac{\log x}{\log\log x}-G^*(k)}$$
is true. Now, notice that $R_\infty(x) \ll \exp(\frac{2\log x}{\log\log x})$, as
$$R_\infty(x)\leq x^{\frac{1}{\log\log x-2\log\log\log x}}R(\log\log x-2\log\log\log x)=$$ $$=\exp\left(\frac{\log x}{\log\log x-2\log\log\log x}+\frac{\log x}{\log\log^2 x}\right).$$
Furthermore, for large enough $x$ the function $G(x)$ is differentiable and convex. Hence, the maximum of the quantity $kx-G(x)$ is attained in the unique point $x_k$ with $G'(x_k)=k$. Consequently,
$$G^*(k) \leq kx_k.$$
Let us now estimate the quantity $x_k$.

It is easy to see that
$$G'(z)=G(z)\frac{z^{-(\varepsilon-2\delta)/(2+\varepsilon-2\delta)}}{1+\varepsilon/2-\delta}$$
Therefore,
$$\exp\left(\exp\left(\frac{\log x_k}{1+\varepsilon/2-\delta}\right)\right)=k(1+\varepsilon/2-\delta)x_k^{1-\frac{2}{2+\varepsilon-2\delta}}.$$
Taking the logarithms and using the fact that $\log x_k=O(\log\log k)$, we get
$$\log x_k=(1+\varepsilon/2-\delta)\log\log k+O\left(\frac{\log\log k}{\log k}\right).$$
Consequently,
$$x_k=(\log k)^{1+\varepsilon/2-\delta}\left(1+O\left(\frac{\log\log k}{\log k}\right)\right).$$
Choose now $k=\left[\frac{\log x}{(\log\log x)^{2+\varepsilon/2-\delta}}\right]=\frac{\log x}{(\log\log x)^{2+\varepsilon/2-\delta}}+O(1)$.

Then we have
$$(\log x)\frac{\log\log\log x}{\log\log x}+O\left(\frac{\log x}{\log\log x}\right)$$
$$ \geq \log x-(2+\varepsilon/2-\delta/2)(\log x) \frac{\log\log\log x}{\log\log x},$$
thus,
$$x^{(\sigma(x)+1)/2}\sqrt{\log x}R_\infty(2x) \geq \sup_{n \leq x} a_n \geq \exp\left(\log x-(2+\varepsilon-\delta/2)\log x\frac{\log\log\log x}{\log\log x}\right).$$
On the other hand, we have
$$\sqrt{\log x}R_\infty(2x) \ll \exp\left(\frac{\delta}{3}(\log x)\frac{\log\log\log x}{\log\log x}\right).$$
So,
$$\exp\left(\log x-(2+\varepsilon-\delta/2)(\log x)\frac{\log\log\log x}{\log\log x}\right)\ll x^{(\sigma(x)+1)/2}\sqrt{\log x}R_\infty(2x) $$
$$\ll
\exp\left(\log x-(2+\varepsilon-\delta/3)(\log x)\frac{\log\log\log x}{\log\log x}\right),$$
and hence
$$\exp\left(\frac{\delta}{6}(\log x)\frac{\log\log\log x}{\log\log x}\right)\ll 1,$$
which is a contradiction.
This concludes the proof for the first case of our theorem.

Now proceed to the case $F(t)=\exp\left(\exp\left(\frac{\log\log t}{\log\log\log t}\right)\right), \sigma(t)=1-\frac{2+o(1)}{\log\log\log t}$.
As before, assume the contrary and consider the function $f(s)=M_G(\zeta^{(n)}(s)+1)$ with $G(z)=\exp((\log z)\log\log z)$ (we extend this function monotonically from $[e,+\infty]$ to all positive real numbers). As the inequality from Theorem 1 is false, we have $\zeta^{(n)}(s) \leq cF(t)$ in the region $\sigma\geq\sigma(t)$, $t \geq t_0$. Thus, in this subset for large enough $t$ the inequality
$$f(s)=M_G(\zeta^{(n)}(s)+1) \leq e^{G(\log_+(|\zeta^{(n)}(s)|+1))} \leq e^{G(\log  F(t))}$$
holds. But $\log F(t)=\exp\left(\frac{\log\log t}{\log\log\log t}\right)$, $\log\log F(t)=\frac{\log\log t}{\log\log\log t}$ and $\log\log\log F(t)\leq \log\log\log t$, so
$$G(\log F(t))=\exp(\log\log F(t)\log\log\log F(t))\leq \exp\left(\frac{\log\log t}{\log\log\log t}\log\log\log t\right)=\log t.$$
Thus,
$$f(s) \ll e^{G(\log F(t))} \leq t.$$
Furthermore, we have $f(s) \ll R\left(\frac{1}{|s-1|}\right)$ for $s \to 1$. Consequently, by the Theorem 2 we have
$$\sup_{n \leq x} a_n \ll x^{(\sigma(x)+1)/2}\sqrt{\log x}R_\infty(2x),$$
where $f(s)=\sum\limits_{n=1}^{+\infty} a_n n^{-s}$.

On the other hand, due to the Theorem 3 the inequality
$$\sup_{n \leq x} a_n \gg \exp\left(\log k \frac{\log x}{\log\log x}-G^*(k)\right)$$
holds. It remains to get an estimate for $G^*(k)$ and choose $k$ optimally. As before, the function $G(x)$ is differentiable and convex for all sufficiently large $x$, therefore
$$G^*(k) \leq kx_k.$$
But $$G'(z)=\left(\frac{\log\log z}{z}+\frac{1}{z}\right)\exp(\log z\log\log z),$$ hence, $$\log x_k=O\left(\frac{\log k}{\log\log k}\right).$$ Taking the logarithms of the both sides of the relation $G'(x_k)=k$, we deduce
$$\log x_k \log\log x_k+\log(\log\log x_k +1)-\log x_k=\log k.$$
From here it is easy to see that
$$\log x_k=\frac{\log k}{\log\log k}+O\left(\frac{\log k}{(\log\log k)^{3/2}}\right).$$
Now choose positive integer $k$ such that $$k=\log x\exp\left(-\frac{\log\log x}{\log\log\log x}+\frac{\log\log x}{(\log\log\log x)^{4/3}}\right)+O(1).$$ Then $$kx_k\leq\log x\exp\left(-\frac12\frac{\log\log x}{(\log\log\log x)^{4/3}}\right)$$ and $$\log k\frac{\log x}{\log\log x}=\log x-\frac{\log x}{\log\log\log x}+O\left(\frac{\log x}{(\log\log\log x)^{4/3}}\right).$$ Therefore, if $\sigma(x)=1-\frac{2}{\log\log\log x}-\frac{c}{(\log\log\log x)^{4/3}}$ for some large enough positive $c$, then
$$x^{(\sigma(x)+1)/2}\sqrt{\log x}R_\infty(2x)=o\left(\exp\left(\log k\frac{\log x}{\log\log x}-G^*(k)\right)\right),$$
as $$\sqrt{\log x}R_\infty(2x)=O\left(\exp\left(\frac{c\log x}{3(\log\log\log x)^{4/3}}\right)\right)$$ and
$$x^{(\sigma(x)+1)/2}=\exp\left(\log x-\frac{\log x}{\log\log\log x}-\frac{c\log x}{(\log\log\log x)^{4/3}}\right).$$
But this relation cannot hold, because
$$x^{(\sigma(x)+1)/2}\sqrt{\log x}R_\infty(2x) \gg \sup_{n \leq x} a_n \gg \exp\left(\log k\frac{\log x}{\log\log x}-G^*(k)\right).$$
A contradiction.

Let us now consider the case when $F$ is a power of double logarithm. We will assume that $F(t)=(\log\log t)^A$, $A>n+1$ and $\sigma(t)=1-(2+o(1))\frac{\log\log\log t}{\log\log t}$. Choose $G(z)=\exp\exp\left(\frac{z}{A}\right)$. If the Theorem 1 for the pair $(F,\sigma)$ is false, then for $\sigma\geq \sigma(t)$ the inequality $|\zeta^{(n)}(s)| \leq c(\log\log t)^A$ holds. Denote $f(s)=M_G(1+\zeta^{(n)}(s))$. As
$$M_G(z)\ll e^{G(\log_+|z|)}=\exp\left(\exp\left(\exp\left(\frac{\log_+|z|}{A}\right)\right)\right)\ll \exp\exp(|z|^{1/A}),$$
we have for $\sigma\geq \sigma(t)$ the inequality
$$|f(s)| \ll \exp\exp(|\zeta^{(n)}(s)+1|^{1/A}) \ll \exp(\exp(\log\log t))=t.$$
Consequently, from the Theorem 2 we find
$$\sup_{n \leq x} a_n \ll x^{(\sigma(x)+1)/2}\sqrt{\log x}R_\infty(2x),$$
because $f(s)=O\left(R\left(\frac{1}{|s-1|}\right)\right)$ for $s \to 1$ (we assumed that $A>n$). On the other hand, the Theorem 3 implies the lower bound
$$\sup_{n \leq x} a_n \gg \exp\left(\log k\frac{\log x}{\log\log x}-G^*(k)\right).$$
Now, as before we need an upper bound for $G^*(k)$ and an optimal choice for $k$. Once again, $G(x)$ is convex for large enough  $x$, so
$$G^*(k) \leq kx_k.$$
Furthermore, we have $$G'(x)=\frac{1}{A}\exp\left(\frac{x}{A}\right)\exp\exp\left(\frac{x}{A}\right) \geq G(x)$$  for all $x \geq A\log A$. Therefore, $$x_k \leq G^{-1}(k)=A\log\log k$$ for all sufficiently large $k$. Set $$k=\left[\frac{\log x}{\log\log x\log\log\log x}\right].$$ Then $$x_k \leq A\log\log k=O(\log\log\log x),$$ hence, $$\log k\frac{\log x}{\log\log x}-G^*(k)=\log x-\frac{\log x\log\log\log x}{\log\log x}+O\left(\frac{\log x}{\log\log x}\right).$$ Thus, if $$\sigma(x)=1-\frac{2\log\log\log x}{\log\log x}-\frac{c}{\log\log x},$$ where $c$ is large enough, then we get a contradiction, because
$$\sqrt{\log x}R_\infty(2x)=O\left(\exp\left(\frac{2\log x}{\log\log x}\right)\right)$$
and
$$x^{(\sigma(x)+1)/2}=\exp\left(\log x-\log x\frac{\log\log\log x}{\log\log x}-\frac{c\log x}{2\log\log x}\right),$$
therefore,
$$\exp\left(\log x-\log x\frac{\log\log\log x}{\log\log x}-\frac{(c-4)\log x}{2\log\log x}\right) \gg \sup_{n \leq x} a_n \gg$$
$$\exp\left(\log x-\log x \frac{\log\log\log x}{\log\log x}+O\left(\frac{\log x}{\log\log x}\right)\right),$$
which is not the case, as $c$ is arbitrarily large.

It remains to examine the case $$F(t)=\exp\exp((\log\log t)^\alpha),\,\sigma(t)=1-\frac{2+o(1)}{(\log\log t)^{1-\alpha}},\,0<\alpha<1.$$
Choose $G(z)=\exp((\log z)^{1/\alpha})$. If the Theorem 1 is not true, then $|\zeta^{(n)}(s)| \leq cF(t)$ for $\sigma\geq \sigma(t)$. Let $f(s)=M_G(1+\zeta^{(n)}(s))$. Then we have
$$|f(s)| \ll \exp(G(\log_+|\zeta^{(n)}(s)+1|)) \ll \exp(G(\log F(t))).$$
 But $\log F(t)=\exp((\log\log t)^\alpha)$, $\log\log F(t)=(\log\log t)^\alpha,$ thus,
$$G(\log F(t))=\exp((\log\log F(t))^{1/\alpha}))=\exp(\log\log t)=\log t,$$
therefore
$$|f(s)| \ll t$$
for $\sigma>\sigma(t)$. As in the previous cases, for $s \to 1$ the inequality
$$|f(s)|=O\left(R\left(\frac{1}{|s-1|}\right)\right),$$
holds. Consequently, the conditions of the Theorem 2 are satisfied and thus,
$$\sup_{n \leq x} a_n \ll x^{(\sigma(x)+1)/2}\sqrt{\log x}R_\infty(2x),$$
where $a_n$ are Dirichlet coefficients of $f(s)$. But the Theorem 3 gives us the lower bound for the same quantity: for any positive integer $k$ we have
$$\sup_{n \leq x} a_n \gg \exp\left(\log k \frac{\log x}{\log\log x}-G^*(k)\right).$$
As always, $G(x)$ is convex for large $x$, therefore
$$G^*(k) \leq kx_k.$$
Differentiating $G$, we find
$$\exp((\log x_k)^{1/\alpha})=\alpha kx_k (\log x_k)^{1-1/\alpha},$$
so
$$(\log x_k)^{1/\alpha}=\log k+\log x_k +O(\log\log k)$$
and
$$\log x_k=(\log k)^\alpha+O((\log k)^{2\alpha-1}).$$
Now, set $k=\left[\log x\exp(-(\log\log x)^{\alpha}-(\log\log x)^{(3\alpha-1)/2})\right]$. Then
$$kx_k=\exp(\log k+\log x_k)=\exp(\log\log x-(\log\log x)^{(3\alpha-1)/2}+O((\log\log x)^{2\alpha-1}))$$
and
$$\log k=\log\log x-(\log\log x)^{\alpha}-(\log\log x)^{(3\alpha-1)/2}+O(1),$$
thus,
$$\sup_{n \leq x} a_n \gg \exp\left((\log k)\frac{\log x}{\log\log x}-G^*(k)\right)\gg \exp\left(\log x-\frac{\log x}{(\log\log x)^{1-\alpha}}(1+o(1))\right).$$
Consequently, the choice $$\sigma(t)=1-\frac{2+o(1)}{(\log\log x)^{1-\alpha}}$$ leads us to the contradiction with the upper bound, which concludes the proof of the Theorem 1.
\end{proof}

\section{Conclusion}

So, with the help of the theorems 2 and 3 we managed to prove a number of omega-theorems for the Riemann zeta function and its derivatives in the regions of the critical strip near the line $\mathrm{Re}\,s=1$. The level of generality of the theorems 2 and 3 also allows to prove omega-theorems for other $L-$functions with nonnegative coefficients. For example, using the Chebotarev density theorem one can prove an analogue of the Theorem 3 which applies to the Dedekind zeta functions of number fiels. Unfortunately, our methods do not provide any nontrivial results about the domains of the form $\sigma\geq \sigma(t)$ with $\sigma(t)=1-o\left(\frac{\log\log\log t}{\log\log t}\right)$ and thus, to prove omega-theorems in this domains, some new ideas are needed.
\end{fulltext}

\end{document}